\documentclass[a4paper,11pt]{article}

\usepackage{hyperref}
\usepackage{fullpage}
\usepackage{amsthm,amsmath, amssymb}
\usepackage{palatino,comment}
\usepackage[usenames, dvipsnames]{color}
\usepackage[active]{srcltx}   
\input{xy}
\xyoption{all}
\xyoption{matrix}

\newtheorem{theorem}{Theorem}[section]
\newtheorem{teo}[theorem]{Theorem}

\newtheorem{lem}[theorem]{Lemma}

\newtheorem{prop}[theorem]{Proposition}

\theoremstyle{definition}

\newtheorem{defi}[theorem]{Definition}

\newtheorem{notation}[theorem]{Notation}

\newtheorem{ex}[theorem]{Example}

\newtheorem{rem}[theorem]{Remark}

\DeclareMathOperator{\Ker}{Ker}

\DeclareMathOperator{\id}{id}

\DeclareMathOperator{\Id}{Id}
\DeclareMathOperator{\End}{End}

\DeclareMathOperator{\GL}{GL} 
\DeclareMathOperator{\SL}{SL} 
\DeclareMathOperator{\Sp}{Sp}
\DeclareMathOperator{\SO}{SO}

\DeclareMathOperator{\uend}{\underline{end}}

\DeclareMathOperator{\ev}{ev}

\newcommand{\ot}{\otimes}

\newcommand{\wt}{\widetilde}
\newcommand{\wh}{\widehat}

\def\M{\mathcal M}

\def\SS{\mathcal S}

\def\g{\mathfrak g}

\def\h{\mathfrak h}

\def\b{\mathfrak b}

\def\B{\mathfrak B}
\def\J{\mathfrak J}

\def\I{\iota}
\def\II{\mathcal I}

\def\C{\mathbb C}

\def\Z{\mathbb Z}

\def\N{\mathbb N}
\def\O{\mathcal O}

\def\om{\omega}
\def\To{\Rightarrow}

\def\tr{\mathrm{tr}}

\def\eps{\varepsilon}

\def\Om{\Omega}

\def\Bmanin{B}

\title{Universal quantum (semi)groups
and Hopf envelopes}

\author{Marco Andr\'es Farinati
\thanks{Dpto de Matem\'atica FCEyN UBA - IMAS (Conicet). 
e-mail: mfarinat@dm.uba.ar.
Partially supported by 
UBACyT 2018-2021
``K-teor\'ia y bi\'algebras en \'algebra, geometr\'ia y topolog\'ia''
and
PICT 2018-00858 ``Aspectos algebraicos y anal\'iticos de grupos
 cu\'anticos''.}}

\date{}

\begin{document}
\maketitle

\begin{abstract}
We prove that, in case $A(c)$ = the FRT construction of 
a braided vector space $(V,c)$ admits
a weakly Frobenius algebra $\B$ (e.g. if the braiding 
is rigid and its Nichols algebra is finite dimensional), then
the Hopf envelope of $A(c)$ is simply the localization
of $A(c)$ by a single element called the quantum determinant 
associated with the weakly Frobenius algebra. This generalizes
a result of the author together
with Gast\'on A. Garc\'ia
in \cite{FG}, where the same statement was proved,
 but with extra hypotheses that we now know 
were unnecessary.
On the way, we
describe a universal way of constructing a universal bialgebra
 attached
to  a finite dimensional vector space together with some
algebraic structure given by a family of maps $\{f_i:V^{\ot n_i}
\to V^{\ot m_i}\}$.
The Dubois-Violette and Launer Hopf algebra and the
co-quasi triangular property of  the FRT construction
play a fundamental role on the proof. 
\end{abstract}

\section*{Introduction}

Given $c:V\ot V\to V\ot V$ a solution of the braid equation, or
equivalently, $R:V\ot V\to V\ot V$ a solution of the 
Yang-Baxter equation, the FRT  construction
(Faddeev-Reshetikhin-Takhtajan) produces a 
coquasitraingular bialgebra
 $A(c)$, so  that its comodule category is naturally braided,
$V$ is  a comodule over $A(c)$, and the map $c$ is recovered
as the categorical braiding. The FRT construction gives
a standard way of constructing quantum semigroups, it is a
 bialgebra that is
never a Hopf algebra (unless the trivial case $V=0$), and the problem
of getting a Hopf algebra by inverting a quantum determinant
 is a classical one, for instance, this problem
is present in Manin's work \cite{M}. In \cite{FG} we give a partial answer,
motivated by the theory of finite dimensional Nichols algebras,
we could exhibit very explicit examples generalizing quantum 
grassmannian algebras and other similar approaches to 
quantum determinants.
The adjective ``explicit'' in \cite{FG} is double: we 
give explicit
formulas for the quantum determinant and explicit formulas
for the antipode.
 However, the main results in \cite{FG} has
hypothesis of two kind: the first main result has a theoretical
assumption, that we know is not always satisfied, but that lead to 
a general statement, and the second main result has an {\em ad hoc}
hypothesis, that is easy to check -from a computational point of view-
in concrete examples, but we couldn't give a general statement where
that hypothesis holds.  This situation is solved in this paper.

We emphasis that in \cite{FG} we give a  framework that generalizes various previous situations, such as
quantum grassmanian algebras (qga) or Frobenius quantum spaces (Fqs) introduced by Manin \cite{M,M2},
 the quantum determinants constructed by Hayashi \cite{H} for  multiparametric quantum deformations of
$\O(\SL_n)$, $\O(\GL_n)$, $\O(\SO_n)$, $\O(O_n)$ and $\O(\Sp_{2n})$, the quantum exterior algebras (qea) in the work of
Fiore \cite{F} for $\SO_q(N)$, $Oq(N)$, and $S_{pq}(N)$. Also,  qea's
 appear in thework of Etingof, Schedler and Soloviev \cite{ESS}.
All these qga’s Fqs and qea's, defined and considered above are quadratic algebras, and in \cite{FG} we
give examples admitting no quadratic qga's Fqs's nor qea's, but still there might 
be a finite-dimensional Nichols algebras associated with it,
hence quantum determinants, and calculation of Hopf envelopes.

It is convenient to see the FRT construction 
 as a bialgebra satisfying a universal
 property, and by doing that, by the same price, one can define
universal bialgebras for many natural and interesting situations, 
including the Hopf algebra introduced by Dubois-Violette and Launer,
and many others.
This construction is so natural that can be considered as folklore,
it is  almost present in Dubois-Violett and Launer's original
 work, but in the author's opinion, there is no
explicit description in the literature, and since this universal
point of view will be used intensively, we present it as the first section
 of this work. We notice that a recent work \cite{Om}
deals with the same universal object,
 focused in the existence problem, but in our context,
the existence problem is trivial and
also our presentation is different and very explicit so decide to keep
that section. Also, in \cite{HWWW} a similar situation is considered, but in
the weak-bialgebra context.

The paper is organized as follows: in {\bf  Section 1}
we define the universal bialgebra associated with a map 
$f:V^{\ot n_1}\to V^{\ot n_2}$ where $V$ is a finite dimensional
vector space. Obvious generalization for families of maps is also presented.

In {\bf  Section 2} we recall the Hopf envelope of a bialgebra and
study in detail the case of the Hopf algebra associated with a
non-degenerate bilinear form. 

In {\bf  Section 3} we
 recall the main construction in \cite{FG}:  the definition
of  Weakly Graded Frobenius Algebra (WGFA) 
and the corresponding candidate
for the explicit formula of the antipode.
After that, using  the coquasitriangular property of the FRT 
construction and the universal construction developed in Section
1  for a modification of the
the Dubois-Violete and Launer's Hopf algebra, the main
result (Theorem \ref{main}) is proved:

\

{\bf Theorem:}  {\em If the FRT construction
admits a Weakly Graded Frobenius Algebra (WGFA, see Definition
\ref{def:finite-Nichols-type}), then its Hopf envelope
is the localization with respect to a single element,
 the quantum determinant associated with the WGFA.}

\

In {\bf  Section 4} we briefly comment other applications
of the universal idea of the first section, and comparaison
with other  works.


{\bf Acknowledgements:} I wish to thank Gast\'on A. Garc\'ia
for fruitful discussions on the developments of this work.
Research Partially supported by the projects
UBACyT 2018-2021
``K-teor\'ia y bi\'algebras en \'algebra, geometr\'ia y topolog\'ia''
and
PICT 2018-00858 ``Aspectos algebraicos y anal\'iticos de grupos
 cu\'anticos''.

\section{Universal construction}
Let $k$ be a field, $V$ a finite dimensional vector space of dimension $n$, let
 $n_1,n_2\in\N_0$, and
\[
f:V^{\ot n_1}\to V^{\ot n_2}
\]
 a linear map. Fix $\{x_i\}_{i=1}^n$ a basis of $V$ and consider
$C$ the coalgebra with basis $\{t_i^j\}_{i,j=1}^n$ and comultiplication
\[
\Delta(t_i^j)=\sum_{k=1}^nt_i^k\ot t_k^j
\]
We consider $V$ as $C$-comodule via
\[
\rho(x_i)=\sum_jt_i^j\ot x_j
\]
Denote $TC$ the tensor algebra on $C$, with bialgebra structure extending
the comultiplication of $C$. 
Since $V$ is a $C$-comodule, then it is a $TC$-comodule, but because $TC$ is a bialgebra,
it follows that $V^{\ot \ell}$ is a $TC$-comodule
for any $\ell\in\N_0$. The structure map is as follows:

For multi-indices $I,J\in\{1,2,\dots,n\}^\ell$ (i.e. $I=(i_1,i_2,\cdots,i_\ell)$)
denote
\[
x_I:=x_{i_1}\ot x_{i_2}\ot\cdots \ot x_{i_\ell}\in V^{\ot \ell}
\]
\[
t_I^J:=t_{i_1}^{j_1}t_{i_2}^{j_2}\cdots t_{i_\ell}^{j_\ell}\in TC
\]
In this notation, the $TC$-comodule structure of $V^{\ot \ell}$
 is given by
\[
\rho(x_I)=\sum_{J\in\{1,\dots,n\}^\ell}t_I^J\ot x_J
\]
If $f(x_I)=\sum_{J}f_I^Jx_J$, it need not be $TC$-colinear,
the condition 
$\rho(f(x_I))=(\id\ot f) \rho(x_I)$ is precisely the commutativity of the following diagram
\[
\xymatrix@-1ex{
x_I\ar@{|->}[d]\ar@/^3ex/@{|->}[rrr]&V^{\ot n_1}\ar[r]^f\ar[d]^{\rho}& 
V^{\ot n_2}\ar[d]^{\rho}&\sum_{J}f_I^Jx_J\ar@{|->}[d]\\
\sum_{J}t_I^J\ot x_J
\ar@/_3ex/@{|->}[drrr]
&TC\ot V^{\ot n_1}\ar[r]^{\id\ot f}& TC\ot V^{\ot n_2}&
\sum_{J,K}f_I^Jt_J^K\ot x_K \\
&&&\sum_{J,K}t_I^Jf_J^K\ot x_K
}
\]
That is, $f$ is colinear if and only if
\[
\sum_{J,K}t_I^Jf_J^K\ot x_K
=\sum_{J,K}f_I^Jt_J^K\ot x_K
\hskip 1cm(\forall I,K)\]
So, we define the two-sided ideal $\II_ f:=\big\langle
\sum_{J}(t_I^Jf_J^K-f_I^Jt_J^K ) : \ \forall I,K\big\rangle$
and the algebra
\[
A(f):=TC/\II_f
\]

\begin{teo}
$A(f)$ is a bialgebra.
\end{teo}
More precisely,  $\II_f\subseteq\Ker(\epsilon)$, where
$\epsilon:TC\to k$ is the algebra map determined by
 $\delta(t_i^j)=\delta_i^j$,
 and
$\Delta\II_f\subseteq
\II_f\ot TC+TC\ot \II_f$. In other words, $\II_f$ is a bi-ideal.

\begin{proof}
Recall, for  $t_I^J=t_{i_1}^{j_1}\cdots t_{i_a}^{j_a}$, the comultiplication is given by
\[
\Delta(t_I^J)=\Delta(t_{i_1}^{j_1})\cdots \Delta(t_{i_a}^{j_a})
\]
\[
=(\sum_{\ell_1}t_{i_1}^{\ell_1}\ot t_{\ell_1}^{j_1})\cdots 
(\sum_{\ell_a}t_{i_a}^{\ell_a}\ot t_{\ell_a}^{j_a})
=\sum_{\ell_1,\dots,\ell_a}
t_{i_1}^{\ell_1}\cdots t_{i_a}^{\ell_a}\ot 
 t_{\ell_1}^{j_1}\cdots t_{\ell_a}^{j_a}
=\sum_{L\in\{1,\dots,n\}^a}t_I^L\ot t_L ^J
\]
So,
\[
\Delta\left(\sum_{J}(t_I^Jf_J^K-f_I^Jt_J^K )\right)
=
\sum_{J,L}
(t_I^L\ot t_L^Jf_J^K-f_I^Jt_J^L\ot t_L^K )
\]
\[=
\sum_{J,L}
(t_I^L\ot (t_L^Jf_J^K - f_L^Jt_J^K )
+t_I^L f_L^J\ot t_J^K 
-f_I^Jt_J^L\ot t_L^K )
\]
\[=
\sum_{J,L}
(t_I^L\ot (t_L^Jf_J^K - f_L^Jt_J^K )
+t_I^L f_L^J\ot t_J^K 
-f_I^Lt_L^J\ot t_J^K )
\]
\[=
\sum_{L}
t_I^L\ot ( \sum_Jt_L^Jf_J^K - f_L^Jt_J^K )
+\sum_J(\sum_L t_I^L f_L^J-f_I^Lt_L^J)\ot t_J^K 
\]
Also
\[
 \epsilon\left(
\sum_Jt_L^Jf_J^K - f_L^Jt_J^K \right)
=\sum_J\delta_L^Jf_J^K - f_L^J\delta_J^K 
=f_L^K - f_L^K=0
\]
We conclude that the ideal generated by
$ \sum_Jt_L^Jf_J^K - f_L^Jt_J^K $ is  a coideal, contained in $\Ker\epsilon$.
\end{proof}

\begin{rem}
{\bf Variation: families of maps.}
The above construction generalizes easily for families of maps.
If 
$F:=\{f_i:V^{n_i}\to V^{m_i}\}_{i\in I}$ is a family of maps, then $\II_F:=\sum_{i\in I}\II_{f_i}$ is a sum of bi-ideals, so, it is a bi-ideal and $A(F):=TC/\II_F$
is a bialgebra
\end{rem}

\subsection{Universal property}

\begin{prop}\label{propuniversal}
The construction $A(F)$ satisfies the following universal property: 
If $V$ is 
a comodule over a bialgebra $A$, with structure map
$\rho_A:V\to A\ot V$, and
 $F=\{f_i:V^{\ot n_i}\to V^{\ot m_i}\}_{i\i I}$ is a family of
$A$-colinear maps,
then there exists a unique bialgebra morphism $\pi:A(F)\to A$ such that
the $A$-comodule structure on $V$ is the one comming from $A(F)$ via $\pi$, that
is, the following diagram is commutative
\[
\xymatrix{
V\ar[r]^{\rho_A}\ar[rd]_{\rho}&A\ot V\\
&A(F)\ot V\ar@{..>}[u]_{\pi\ot \id_V}
}\]
\end{prop}
\begin{proof}
Fix a basis $x_1,\dots x_n$ of $V$ then $\rho_A(x_i)=\sum_ja_i^j\ot x_j$
for a unique choice of elements $\{a_i^j\}_{i,j}$ in $A$. Define the map 
$\pi:A(F)\to A$ via $t_i^j\mapsto a_ i ^j$, since all $f_i$ are
 $A$-colinear it follows 
that $\pi$ is well defined.
\end{proof}

\begin{ex}\label{dg}
Let $V=k[x]/x^2$ considered as unital algebra. That is, we have two maps,
the multiplication and the unit:
\[
m:V\ot V\to V\]
\[
u:k=V^{\ot 0}\to V\]
Consider $\{1,x\}$ as a basis of  $k[x]/x^2=k\oplus kx$. If
\[
\rho(1)=a\ot 1+b\ot x\]
\[
\rho(x)=c\ot 1+d\ot x\]
the condition of $u$ being co-linear, since $\rho(1_k)=1\ot 1_k$
forces $a=1$ and $b=0$. Because
 $x^2=0$,
\[
(\id\ot m)(\rho(x\ot x))=c^2\ot m( 1\ot 1) + cd\ot m(1\ot x) + dc\ot m(x\ot 1)
+d^2\ot m(x\ot x) 
\]
\[
=c^2\ot 1 + (cd+dc)\ot x 
\]
we have in $A(m)$ the relations $c^2=0$ and $dc=-cd$. Notice the (co)matrix
comultiplication
\[
\Delta a=a\ot a+b\ot c,\ 
\Delta b=a\ot b+b\ot d \]
\[
\Delta c=c\ot a+d\ot c,\ 
\Delta d=c\ot b+d\ot d \]
with $a=1$ and $b=0$ gives
\[
\Delta c=c\ot 1+d\ot c,\ 
\Delta d=d\ot d \]
That is, $d$ is group-like and $c$ is a skew primitive. We conclude
\[
A=A(\{m,u\})=k\{c,d\}/\langle c^2, cd+dc\rangle=k[\N_0]\# k[c]/c^2
\]
A comodule
structure over $k[\N_0]\# k[c]/c^2$ is
precisely a d.g. structure, in this case, this construction gives
the natural grading $|1|=0$, $|x|=1$, together
with  the differential  $\partial x=1$, $\partial 1=0$.
Notice that -at least if $\frac12\in k$-, the abelianization
$A_{ab}:=A/([A,A])=k[d]=k[\N_ 0]$. In the ``classical'' setting
one gets only a natural grading on $k[x]/(x^2)$ (or a Torus action),
but in this non-commutative or ``quantum semigroup'' action, one gets
the differential structure due to the element $c$.
\end{ex}

\begin{rem} Usual classical objects that are invariant under the group of
 automorphisms do not
need to be automatically ``quantum invariant''. For example,
if $A$ is a finite dimensional algebra, one can consider the trace
map $\tr:A\to k$ given by 
\[
\tr (a) =\tr\Big( a'\mapsto a\cdot a'\Big)
\] 
In the above example, $\tr(1)=2$ and $\tr(x)=0$, but
\[
(\id\ot\tr)\rho(x)=
(\id\ot\tr)(c\ot 1+d\ot x)=2c \neq 0 = \rho(\tr x)\]
That is, $\tr$ {\em is not} $A(m,u)$-colinear.
Similarly, the Killing form of a finite dimensional Lie algebra $(\g,[,])$
is not necessarily colinear with respect to the algebra $A([,])$.
Nevertheless, one can always add the relation associated with that
operation, form instance, $\tr$ is always colinear with respect to
$A(m,u,\tr)$, and the Killing form will be colinear using $A([,],\kappa)$.
\end{rem}

\section{The Hopf algebra of Dubois-Violette and Launer}

In \cite{DVL}, the authors define a Hopf algebra associated with a
non-degenerate bilinear form in the following way:

Let  $b:V\ot V\to k$ be  a non degenerate bilinear map,
write

\[b(x_i\ot x_j)=b_{ij}\in k\]
and denote $b^{ij}$ the matrix coefficients of the inverse of the
matrix $(B)_{ ij}=b_{ij}$, and as
in the previous section, $\{x_i:i=1\dots,n\}$ is a basis of $V$.
 Dubois-Violette and Launer define
the $k$-algebra with generators $\{t_i^j: i,j=1,\dots,n\}$ and relations
(sum over repeated indices).

\begin{eqnarray}\label{eqdv1}
b_{\mu\nu}t_ \lambda^\mu t_\rho^\nu = b_{\lambda\rho}1
\\
\label{eqdv2}
b^{\mu\nu}t^ \lambda_ \mu t^\rho_ \nu = b^{\lambda\rho}1
\end{eqnarray}

In our setting,  $V$  is a comodule
over the free bialgebra with generators $t_i^j$ and $V^{\ot 2}$
is a comodule via (sum over repeated indices)
\[
\rho(x_i\ot x_j)=
t_i^kt_j^l\ot x_k\ot x_l,\]
Considering $k$ as trivial comodule,
 the colinearity of $b$ requires
\[
\rho(x_i\ot x_j)=1\ot b(x_i\ot x_j)
=1\ot b_{ij}\overset{?}{=}(\id\ot b)(\rho(x_i\ot x_j))=
t_i^kt_j^l \ot b_{kl}
\]
hence, the equations for $A(b)$ are
\begin{eqnarray}\label{eqdv3}
t_i^kt_j^l b_{kl}=b_{ij}
\end{eqnarray}
This is the same as equation \eqref{eqdv1}.

\

In terms of matrices, 
consider  $B\in M_n(k)\subset M_n(A(b)(c))$ and $\frak t\in M_n(A(b))$ given by
  $(B)_{ij}=b_{ij}$ and $(\frak t)_{ij}=t_i ^j$; the relation
\eqref{eqdv3} is:
\[
 \frak t\cdot B\cdot \frak t=B
\]
or equivalently
\[
B^{-1}\cdot  \frak t\cdot B\cdot \frak t=\Id
\]
because $B$ is invertible in $M_n(k)$, and so it is invertible in $M_n(A(b))$.
We see that $\frak t$ has a {\em left} inverse $B^{-1}\cdot\frak t\cdot B$,
we will show that $\frak t$ also have
 $B^{-1}\cdot\frak t\cdot B$ as {\em right} inverse. For that
denote
\[
U:=\frak t\cdot (B^{-1}\cdot \frak t\cdot B)
\]
We want to show that $U=\id$. We compute
$B^{-1}\frak t BU$ and get
\[
B^{-1}\frak t BU=B^{-1}\frak tB\frak t B^{-1} \frak t B
\]
and using $\frak t \cdot B\cdot \frak t =B$ we get
\[
B^{-1}(\frak tB\frak t) B^{-1} \frak t B
=B^{-1}B B^{-1} \frak t B
= B^{-1} \frak t B
\]
But from $B^{-1}\frak t BU= B^{-1}  \frak t B$ it follows
that
\[
\frak t BU=   \frak t B 
\]
and so
\[
B^{-1}\frak t B\frak t BU=   B^{-1}\frak t B\frak t B
\]
Recall $\frak t B \frak t = B$, so
$B^{-1}\frak t B\frak t=\id$ and we conclude from the above equation that
\[
BU= B
\]
which clearly implies $U=\Id$.

Notice that
 $U=\Id$ means
\[
\frak t\cdot B^{-1}\cdot \frak t\cdot B = \id
\]
or equivalently
\[
\frak t\cdot B^{-1}\cdot \frak t=B^{-1}
\]
and the components of this equation is precisely equation
 \eqref{eqdv2}, so,
 equation \eqref{eqdv2} is redundant
 and Dubois-Violette and Launer
Hopd algebra coincides with the bialgebra $A(b)$. In particular,
$A(b)$ is a Hopf algebra, the antipode is
given by
\[
S(t_i^j)=b^{jk}t_k^lb_{li}
\]
We  mention that the  equations
\[
S(h_1)h_2=\epsilon(h) 
\ \hbox{ and }\
h_1S(h_2)=\epsilon(h)
\]
for  $h=t_i^j$ (and for all $i,j$) mean, respectively,
 \[
(B^{-1}\frak t B) \frak t=\id
 \ \hbox{ and }\
\frak t (B^{-1}\frak t B) =\id
\]

The fact that Dubois-Violette and Launer construction gives a
 Hopf algebra is well-known, but for completeness we include
the following:

{\em Proof that the antipode is well-defined:}
Recall the relation
 $t_i^kt_j^lb_{kl}=b_{ij}$;
we want to see that the opposite relation is valid for $S(t_i ^j)$, that is
\[
S(t_j^l)S(t_i^k)b_{kl}\overset{?}{=}b_{ij}
\]
Let us denote
\[
\wt B_{ji}:=S(t_j^l)S(t_i^k)b_{kl}
=
b^{la}t_a ^cb_{cj}
b^{kd}t_d ^eb_{ei}b_{kl}
\]
\[=
\delta^{a}_kt_a ^cb_{cj}
b^{kd}t_d ^eb_{ei}
=
t_k ^cb_{cj}
b^{kd}t_d ^eb_{ei}
\]
We have
\[
\wt B_{ji}=
t_k ^cb_{cj}
b^{kd}t_d ^eb_{ei}
\]
So
\[
\wt B_{ji}t^i_u=
t_k ^cb_{cj}
b^{kd}t_d ^eb_{ei}t^i_ u
=
t_k ^cb_{cj}
b^{kd}(t_d ^eb_{ei}t^i_ u)\]
\[
=t_k ^cb_{cj}b^{kd}b_{du}\]
\[
=t_k ^cb_{cj}\delta_ u^k=t_u ^cb_{cj}=b_{ij}t^i_ u\]
In matrix notation, $\wt B\frak t=B^t\frak t$.
Since $\frak t$ is invertible, it follows that
$\wt B_{ji}=b_{ij}$ as desired.

\begin{rem}
Other situations where the universal bialgebra is already Hopf,
outside non-degenerate bilinear forms, are possible;
see for instance
\cite{BD} or \cite{CWW}. However, the
 Dubois-Violette and Launer's
Hopf algebra will be enough for our purpose.
\end{rem}

\subsection{The Hopf envelope}

It is well-known  that
the forgetful functor from Hopf algebras to bialgebras has a 
left adjoint, the general construction is due to Takeuchi
\cite{T}.
 In other words, if $B$ is a bialgebra, then there exists
a Hopf algebra $H(B)$ together with a bialgebra map $\I_B:B\to H(B)$
such that every map $f:B\to H$ from $B$ into a Hopf algebra 
$H$ factors in a unique way through $H(B)$:
\[
\xymatrix{
\ar[rd]_ {\I_B}B\ar[rr]^{\forall f}&&H\\
&H(B)\ar@{..>}[ru]_{\exists ! \wt f}
}\]
The general construction can be complicated: if $B$ is given by 
generators $b_1,\dots,b_m$ and relations, one should add
extra generators $b'_1,\cdots, b'_m$ so that
$S(b_i)=b'_i$, and relations in order to get the Antipode axiom,
but also one should add $S^2(b_i)=S(b'_i)$ as generator,
say $b''_i$, and so on. It is not clear in general
if the Hopf envelope
of a finitely generated bialgebra is finitely generated. In some cases,
very few elements are really necessary in order to get a Hopf algebra.
For example, if one knows a priori that $S^2=\id$ (e.g. if $B$ is commutative
or cocommutative), then the double of the original generators 
will be enough. A particularly simple example
is
 $B=O(M_n(k))$, whose Hopf envelope
is $\O(GL_n(k))=\O(M_n(k))[\det^{-1}]$. That is, we only add
a single commuting generator $D^{-1}$, with the relation 
$D^{-1}\cdot \det=\det\cdot  D^{-1}=1$.

In the framework of universal biagebra $A(F)$
associated with a family of maps $F$, one can also consider
$H(F):=H(A(F))$, that is, the Hopf envelope of $A(F)$. It will have the
 analogous universal property as $A(F)$ but within the Hopf algebras.
It is not clear how many generators are really necessary 
to add, but very different things can happen. From this point of view,
the FRT construction (see subsection \ref{sectionFRT}) is
an opposite example of the Dubois-Violette and Launer's Hopf
algebra:
The universal
bialgebra associated with a non degenerate bilinear form is already a Hopf
algebra, while the FRT construction
is never Hopf (unless the trivial case $V=0$).
 However,
the main result of this paper is to show a general circumstance
where the Hopf envelope of the FRT construction is the 
(in general non-commutative - though normal) localization
with respect to a single element, that we call the quantum determinant.

Notice that if $H$ is a Hopf algebra and  $I$ a bi-ideal such that 
$S(I)\subseteq I$, then clearly $H/I$ is a Hopf algebra.
This can be applied to the following situation:

\begin{teo}
Let $V$ be a finite dimensional
vector space and let us fix  a linear
isomorphism $\Phi:V\to V^{**}$ (not necessary the canonical one).
Consider the maps
\[
\ev_l:V^*\ot V\to k\]
\[
\varphi\ot x\mapsto \varphi(x)\]
\[
\ev_r:V\ot V^*\to k\]
\[
x\ot \varphi\mapsto \Phi(x)(\varphi)\]
If $W:=V\oplus V^*$, then the universal bialgebra
on $W$ such that the decomposition
 $W=V\oplus V^*$ and the bilinear maps
$\ev_l$ y $\ev_r$ are colinear, is already a  Hopf
algebra. We will denote it by
 $H(\ev_l,\ev_r)$.
\end{teo}

\begin{proof}
If $b:W\ot W\to k$ is the bilinear map determined
 by $b(v,w)=0=b(\phi,\psi)$,
$b(\phi,v)=\ev_l(\phi,v)$,
$b(v,\phi)=\ev_r(\phi,v)$, then  $b$ is 
clearly non-degenerate; its universal bialgebra is Hopf
because it coincides with Dubois-Violette and Launer's one.
Let us call it $H(b)$.
Let  $x_1,\dots, x_n$ be a basis of  $V$, $x^1,\dots,x^n$ 
its dual basis, that for covinience we will denote
$x_{ n+1},\dots,x_{2n}$.
Recall $H(b)$ has generators $t_i^j:i,j=1,\dots,2n$ and the antipode
is given by
\[
S(t_ i^j)=b^{jk}t_k^lb_{li} 
\]
Since $b(V,V)=b(V^*,V^*)=0$, the matrix $B\in k^{2n\times 2n}$
of $b$ has a structure of $2\times 2$ blocks of size
  $n\times n$ of the form
\[
B=
\begin{pmatrix}
0&*\\
*&0
\end{pmatrix}
\]
Similar block structure for its inverse.
One may write, for $i,j=1,\dots,n$
\[
S(t_ i^{n+j})=\sum_{k,l=1}^{2n}b^{n+j,k}t_k^lb_{li} 
=\sum_{k=1}^{n}\sum_{l=n+1}^n b^{n+j,k}t_k^lb_{li} 
=\sum_{k,l=1}^{n} b^{n+j,k}t_k^{n+l}b_{n+l,i} 
\]
and similarly
\[
S(t_ {n+i}^{j})=\sum_{k,l=1}^{2n}b^{j,k}t_k^lb_{l,n+i} 
=\sum_{k=n+1}^{2n}\sum_{l=1}^{n}b^{j,k}t_k^lb_{l,n+i} 
=\sum_{k,l=1}^{n}b^{j,n+k}t_{n+k}^lb_{l,n+i} 
\]
If we add the condition ``the decomposition $W=V\oplus V^*$
is  colinear'', this is the same as
require that the projector
\[
\pi_V:V\oplus V^*\to V\oplus V^*\]
\[x_i\mapsto x_i,\ \ (i=1,\dots,n)\]
\[
x_{n+i}\mapsto 0\ \ (i=1,\dots,n)\]
should be colinear. (Notice $\pi_{V^*}=\id_W-\pi_V$, so we don't need to add
the other projector to the family of maps).
We have, for
 $i=1,\cdots,n$:
\[
\rho(x_i)=\sum_{j=1}^{2n}t_i^j\ot x_j
=\sum_{j=1}^{n}t_i^j\ot x_j
+\sum_{j=1}^{n}t_i^{n+j}\ot x_{n+j}
\]
\[
\To \ (\id\ot\pi_V)(\rho(x_i))=\sum_{j=1}^{n}t_i^j\ot x_j
\]
\[
\rho(x_{n+i})=\sum_{j=1}^{2n}t_{n+i}^j\ot x_j
=\sum_{j=1}^{n}t_{n+i}^j\ot x_j
+\sum_{j=1}^{n}t_{n+i}^{n+j}\ot x_{n+j}
\]
\[
\To \ (\id\ot\pi_V)(\rho(x_{n+i}))=\sum_{j=1}^{n}t_{n+i}^j\ot x_j
\]
If one ask  $\pi_V$ to be colinear then the relations needed are
\[
t_i^{n+j}=0=t_{n+i}^j \ \ \forall i,j=1,\dots,n
\]
We see from the previous computation that the ideal generated by
$\{t_i^{n+j},t_{n+i}^j, i,j=1,\dots,n\}$ is stable by the antipode and so,
the quotient bialgebra 
\[
A(b,\pi_V)=H(b)/\langle t_i^{n+j}=0=t_{n+i}^j,
 i,j=1,\dots,n  \rangle
\]
 is a Hopf algebra.
Since $\pi_V:W\to W$  is $A(b,\pi_V)$ colinear, then
$\pi_{V^*}=\id_W-\pi_V$
is colinear too, hence,  $V\subset W$ and $V^*\subset W$
are subcomodules. Since $b$ is colinear, we  conclude
that  $\ev_l$ and $\ev_r$ are colinear maps too,
because they can be computed using $b$ and
restrictions from $W$ into subcomodules, that is,
using compositions with colinear inclusions.
\end{proof}

As a corollary, we can give a proof of Theorem \ref{main}, 
that is the same statement as in \cite{FG} but almost without hypothesis.

\section{Main resut}

After recalling the main objects of interest here: the FRT construction and
 Weakly Graded Frobenius Algebras, we prove our main result.
 
\subsection{FRT construction: $A(c)$ \label{sectionFRT}} 

A \textit{braided vector space} is a pair $(V,c)$, where $V$ is $k$-vector space and
$c\in \End(V\ot V)$ is a solution of the braid equation:
\begin{equation}\label{eq:braideq}
(c\ot \id)(\id\ot c)(c\ot \id)=
(\id\ot c)(c\ot \id)(\id\ot c)\qquad \text{ in }\End(V\ot V\ot V),
\end{equation}
In \cite{FRT}, the authors define a bialgebra associated with an $R$-matrix,
but to have an $R$-matrix is equivalent to have a braiding considering 
$c:=\tau\circ R$, where $\tau:V\ot V\to V\ot V$ is the usual flip.
In terms of the matrix coefficients of $c$, the 
$FRT$ construction is the $k$-algebra generated by
$t_i^j$ with relations
\begin{equation}\label{eq:FRT}
\sum_{k,\ell}c_{ij}^{k\ell}t_k^rt_ \ell^s=
\sum_{k,\ell}t_i^kt_ j^\ell c_{k\ell}^{rs}
\hskip 1cm \forall\ 1\leq i,j,r,s\leq n.
 \end{equation}
It turns out that the FRT construction is {\em exactly} $A(c)$.

The fact that $c$ is a solution of the braid equation
implies the very important fact that
$A(c)$ is a co-quasi-triangular
bialgebra. That is, there exists a convolution-invertible bilinear
 map
$r:A\times A\to  \Bbbk$ satisfying
\[
\begin{array}{crcl}
(CQT1)\qquad&r(ab,c)&=&r(a,c_{(1)})r(b,c_{(2)})
\\
(CQT2)\qquad&r(a,bc)&=&r(a_{(2)},b)r(a_{(1)},c)
\\
(CQT3)\qquad&r(a_{(1)},b_{(1)})a_{(2)}b_{(2)}&=&b_{(1)}a_{(1)}r(a_{(2)},b_{(2)})
\end{array}
\]
This map is uniquely determined by
\[
r( t_i^k,t_j^\ell)=
c_{ji}^{k\ell}\qquad \text{ for all }1\leq i,j,k,\ell\leq n.
\]
(Notice the indices $ij$ and $ji$ in the definition of $r$.)
In particular, the category of $A(c)$-comodules is braided.

\subsection{Weakly Graded Frobenius algebras}

In this subsection and the following we recall the main definitions and results of \cite{FG}.
We begin with the definition of weakly
graded Frobenius algebra, that
 extends the notion of Frobenius quantum space 
introduced by Manin in \cite[\S 8.1]{M2}. The motivation
is  to produce quantum determinants together with
 quantum Cramer-Lagrange identities, hence a
formula for the antipode.
The paradigmatic examples are finite dimensional Nichols algebras
 associated with rigid solutions of the braid equation.

\begin{defi}\cite[2.1]{FG}\label{def:finite-Nichols-type} 
Let $\mathcal{A}$ be a bialgebra and $V \in\ ^{\mathcal{A}}\M$.
An  algebra $\B$ is called a
{\bf weakly graded-Frobenius} (WGF) algebra for $\mathcal{A}$ and $V$ if the following conditions are satisfied:
\begin{itemize}
\item[WGF1)] $\B$ is an $\N$-graded $\mathcal{A}$-comodule algebra, that is
 $\B=\underset{n\geq 0}\bigoplus\B^n$, 
$\rho(\B^n)\subseteq \mathcal{A}\ot \B^n$, where
$\rho:\B\to  \mathcal{A}\ot \B$ is the structure map,
and $\B^n\cdot \B^m\subseteq \B^{n+m}$ for all $n,m\geq 0$;
\item[WGF2)]  $\B$ is connected (i.e. $\B^0=\Bbbk$) and $\B^1=V$ as $\mathcal{A}$-comodules;
\item[WGF3)] $\dim_\Bbbk\B<\infty$ and $\dim_\Bbbk\B^{top}=1$, where
$top=\max\{n\in\N : \B^n\neq 0\}$;
\item[WGF4)] the multiplication induces
non-degenerate bilinear maps
\[
\B^1\times\B^{top-1}\to \B^{top},\qquad
\B^{top-1}\times\B^1\to \B^{top}.
\]
\end{itemize}
\end{defi}
We notice that 
conditions in (WGF4) appeared in \cite{M2}, related to involutive
 solutions of the 
QYBE (thus the corresponding $c$ is a symmetry) and in \cite{Gu},
related to Hecke-type solutions. It is known that in both cases the
 quantum
exterior algebras are Nichols algebras, thus this Definition generalizes
 \cite{M2,Gu}.

\begin{defi} Let $\B$ be a WGF algebra for $A$ and write 
$\B^{top}=\Bbbk \b$ for some $0\neq \b\in \B$.  
We call such an element a \textit{volume element} for $\B$.
Since $\B^{top}$  is an $A$-subcomodule, $\rho(\b)= 
D\ot \b$ for some group-like element $D\in A$. 
We call this element $D$ {\em the quantum determinant in $A$ associated with $\B$}.
\end{defi}

\begin{notation}\label{not:Tij}
Let $\{x_1,\dots,x_n\}$ be a basis of $V$. 
Since by assumption  the multiplication $\B^1\times\B^{top-1}\to \B^{top}=\Bbbk \b$ is non-degenerate,
there exists a basis of $\B^ {top-1}$, say
$\{\om^1,\dots,\om^n\}\in\B^ {top-1}$, such that
\[
x_i\om^j=\delta_i ^j \b\quad \in \B^{top}.
\]
For $1\leq i,j\leq n$, we define the elements $T_i^j\in A$ by the equality
\[
\lambda(\om^i)=\sum_j T^i_j\ot \om^j \qquad \text{ for all }1\leq i \leq n.
\]
It is easy to check that  
$\Delta (T_j^i)=\sum_{k=1}^{n}T_k^i\ot T_j^k$ 
and $\eps(T_{i}^{j}) = \delta_{i}^{j}$ for all $1\leq i,j\leq n$.
\end{notation}

\begin{ex} 
If $V=k^n$, $c=-\tau$ on $V\ot V$, $A(c)=\O(M_n(k))$, $\B=\Lambda V$,
then  $\b=x_1\wedge\cdots x_n$ is the usual volume
form, the elements  
$w^j=(-1)^{i+1}x_1\wedge\cdots \wh {x_j}\cdots \wedge x_n$ give a "dual basis"
with respect to $\{x_1,\dots,x_n\}$. The elements
$T_i^j$ are the minors of the generic matrix. The bialgebra
 $\O(M_n(k))$
is not Hopf, but its localization
$\O(GL(n,k))=\O(M_n(k))[\det^{-1}]$ is a Hopf algebra.
\end{ex}

One of the main goals in \cite{FG} was to generalize
the Lagrange formula for expanding the determinant by rows, 
and hence to have a natural candidate for the antipode on the
 localization
by quantum determinants. The general statement is:

\begin{prop}\cite[Proposition 2.6]{FG}
\label{Lagrange}
The following formula holds in $A(c)$:
\begin{equation}\label{eq:A(c)-D}
\sum\limits_{k=1}^n t_i^kT^j_ k = \delta_ i^j D \qquad \text{ for all }1\leq i,j\leq n.
 \end{equation}
\end{prop}

We recall a result of Hayashi. 
\begin{lem}
\cite[Theorem 2.2]{H} Let $A$ be a coquasitriangular bialgebra.
For any group-like element $g\in A$, there is
a bialgebra automorphism $\mathfrak J_{g}:A\to A$ given by 
$
\J_{g}(a)=r(a_{(1)},g)a_{(2)}r^{-1}(a_{(3)},g)$
such that
\[ga=\J_{g}(a)g\qquad \text{ for all } a\in A.\]
\end{lem}

In particular, $D\in  A(c)$ is a group-like element 
in a cqt bialgebra, so we have a bialgebra isomorphism
$\J:A(c)\to A(c)$ such that
\[
Da=\J(a)D\qquad \text{ for all } a\in A(c)
\]

\begin{defi}\label{def:H(c)}
Let $A(c)[D^{-1}]$ be the $\Bbbk$-algebra 
generated by $A(c)$ and a new element $D^{-1}$ satisfying the 
relations
\begin{equation}\label{eq:D}
DD^{-1}= 1= D^{-1}D. 
\end{equation}
 \end{defi}
\noindent 
It easy to see that $A(c)[D^{-1}]$ is indeed a (non commutative)
 localization of $A(c)$ in $D$. We denote by $\iota: A(c)\to A(c)[D^{-1}]$
the canonical map.
 Notice that, in virtue of Hayashi's result,
$D$ is a normal element, so, 
the general non-commutative localization can be computed in terms
of left (or right) fractions: a general element in $A(c)[D^{-1}]$
is of the form $D^{-n}a$ for some $n\in \N$ and $a\in A(c)$.
The next result follows from \cite[Theorem 3.1]{H}, see also
 \cite[Lemma 2.13]{FG}.

\begin{lem}\label{lem:H(c)bialgebra}
$A(c)[D^{-1}]$ is a coquasitriangular bialgebra.
\end{lem}

\subsection{Main result}

In this section
\begin{itemize}
\item  $(V,c)$ is a finite dimensional braided vector space, 
\item $A(c)$ is the FRT construction,
\item  we assume that $A(c)$ admits a WGF algebra
$\B$
(see Definition \ref{def:finite-Nichols-type}),
denote $D$ its associated quantum determinant.
\end{itemize}

Since $D\in A(c)$ is a group-like element, it must be invertible
in the Hopf envelope of $A(c)$.
In \cite {FG} we studied the problem of deciding 
if inverting $D$ is enough in order to get a Hopf algebra.
In other words, to decide if $A(c)[D^{-1}]$ is a Hopf algebra, and 
moreover, to give
 the formula
for the antipode. 
The first main result 
(see \ref{not:Tij} for the notation of the $T_i^j$'s) 
in \cite{FG} is:

\begin{teo}\cite[Theorem 2.19]{FG})
If the canonical map $\iota: A(c)\to A(c)[D^{-1}]$ is injective
 then the category of $A(c)[D^{-1}]$-comodules
is rigid, tensorially generated by
$V$ and $\Bbbk D^{-1}$.
As a consequence, $A(c)[D^{-1}]$ is a 
 Hopf algebra. Moreover, the formula for the antipode is
given on generators by
$\SS(D^{-1})=D$, and 
\[
\SS(t_i^j):=T_i^jD^ {-1}
\hskip 1cm
(\forall 1\leq i,j\leq n).
\]
\end{teo}
With the notation of the automorphism $\J=\J_D$ given by Hiyashi's
theorem, the second
 main results in \cite{FG} is:

\begin{teo}\cite[Theorem 2.21]{FG})
Assume the following equality holds in $A(c)$ for all $1\leq i,j\leq n$:
\begin{equation}\label{eq:teocompfirendly}
\sum\limits_{k=1}^{n}\J(T_i^k)t^j_ k = \delta_ i^j D.
\end{equation}
Then $A(c)[D^{-1}]$ is a  Hopf algebra and the formula for the antipode  on generators is given by
$\SS(D^{-1})=D$, and
$
\SS(t_i^j):=T_i^jD^ {-1}
$ for all $1\leq i,j\leq n$.
\end{teo}

Now, without assuming $\iota: A(c)\to A(c)[D^{-1}]$
 (there are examples
with generic
braidings of diagonal type
where $\iota$ is not injective,
see last example in \cite{FG}), nor asuming
 \eqref{eq:teocompfirendly},
we can prove the main result of this paper:

\begin{teo}\label{main}
$A(c)[D^{-1}]$ is always a Hopf algebra.
\end{teo}
\begin{proof}
First notice that equations
\eqref{eq:A(c)-D} 
(Proposition \ref{Lagrange})
and
\eqref{eq:teocompfirendly} are equations in $A(c)$, but in $A(c)[D^{-1}]$
one can write respectively as

\begin{equation}\label{facil}
\sum\limits_{k=1}^n t_i^kT^j_ kD^{-1} = \delta_ i^j 
 \end{equation}
and
\begin{equation}\label{delotrolado}
\delta_ i^j =
\sum\limits_{k=1}^{n}D^{-1}\J(T_i^k)t^j_ k 
=
\sum\limits_{k=1}^{n}T_i^kD^{-1}t^j_ k 
\end{equation}
We know equation \eqref{facil} is true because 
of \eqref{eq:A(c)-D}, and it means that defining
$S(t_i^j)=T_i^jD^{-1}$, in case $S$ is well-defined and 
antimultiplicative, it will
 satisfy the antipode axiom on the right. In \cite{FG} we show that 
\eqref{delotrolado} implies that $S$ is actually well-defined as 
antialgebra
map, and that also
$S$ satisfies the antipode axiom on the left.
 So,
 it will be enough to show that equation
\eqref{eq:A(c)-D} implies equation \eqref{delotrolado}.

Using that $A(c)[D^{-1}]$ is coquasitriangular,
the category of comodules is braided, and so, the following two
comodules are isomorphic:
\[
V^*:=\B_{top-1}\ot k D^{-1}
\]
\[
{}^*V:=kD^{-1}\ot \B_{top-1}
\]
Fix an isomorphism.
Define $W:={}^*V\oplus V$
and $b:W\times W\to k$ the
bilinear map as follows:
\[
b(v,v')=0=b(\phi,\phi') \hskip 1cm \forall\ \  v,v'\in V, 
\phi,\phi'\in {}^*V
\]
Define $b(\phi,v)$ through the multiplication
$m:\B_{top-1}\ot V\to \B_{top}$:
\[
{}^*V\otimes V=kD^{-1}\ot \B_{top-1}\otimes V
\overset{\id\ot m}{\longrightarrow}
kD^{-1}\ot \B_{top}=kD^{-1}\ot kD\cong k
\]
and finally $b(v,\phi)$
through the fixed isomorphism ${}^*V\cong V^*$:

\[
V\ot {}^*V=V\ot (kD^{-1}\ot\B_{top-1}) 
\cong  V\ot V^*=
V\ot (\B_{top-1}\ot D^{-1})
\overset{m\ot \id}{\longrightarrow}
k D\ot kD^{-1}
\cong 
k
\]
It is clear that $b$ is non degenerate, so $H(b)$ is a Hopf algebra.
But also, because the multiplication in $\B$ is $A(c)$ colinear, and
the isomorphism $kD^{-1}\ot kD\cong k$ is $A(c)[D^{-1}]$ colinear,
we get that $b$ is also $A(c)[D^{-1}]$ colinear. Using the universal property for $H(b)$ we get a map
\[
H(b)\to A(c)[D^{-1}]
\]
Moreover, $V^*$ and $V$ are $A(c)[D^{-1}]$ comodules, hence
the projection $\pi_V:V\oplus V^*\to V$ is colinear, and the above epimorphism factor through
\[
H(b)\to H(\ev_l,\ev_r)\to A(c)[D^{-1}]
\]
Recall  the set of generators $\{t'{}_i^j, i,j=1,\dots,2n\}$ of
$H(b)$ and $H(\ev_l,\ev_r)$. Notice the subset
 $\{t'{}_i^j, i,j=1,\dots,n\}$ maps into the generators
$\{t_i^j, i,j=1,\dots,n\}$ of
 $A(c)$.

We also know from
Proposition \ref{Lagrange}
that the following equation holds:
\[\sum\limits_{k=1}^n t_i^kT^j_ k = \delta_ i^j D \qquad \text{ for all }1\leq i,j\leq n.
\]
so, in $A(c)[D^{-1}]$ we have

\begin{equation}\label{eqid}
\sum_{k=1}^n t_i^kT_k^jD^{-1}=\delta_i^j
\end{equation}
or in matrix notation
\[
\frak t\cdot  \frak T=\id_{n\times n}
\]
where $(\frak t)_{ij}=t_i^j$ and  $(\frak T)_{ij}=T_ i^jD^{-1}$
are elements of $M_n(A(c)[D^{-1}])$.

Now one can compute in $M_n(H(b))$
the equality
given by the antipode property, for $i,j=1,\dots,n$:
\[
\delta_{i}^j=
\epsilon(t'{}_i^j)=(t'{}_i^j)_1S((t'{}_i^j)_2)=\sum_{k=1}^{2n}t'{}_i^kS(t'{}_k^j)
=\sum_{k,l,r=1}^{2n}t'{}_i^k b^{kl}t'{}_l^rb_{rj}
\]
But in $H(\ev_l,\ev_r)$, using $t_a^{n+b}=0=t_{n+a}^b$,
 the above sum gives
\[
=\sum_{k=1}^{n}\sum_{l,r=1}^{2n}t'{}_i^k b^{kl}t'{}_l^rb_{rj}
\]
And because the only possible nonzero coefficients of the bilinear
form are $b_{k,n+l}$, or $b_{n+k,l}$ ($k,l=1,\dots,n$) this 
is the same as
\[
=\sum_{k=1}^{n}\sum_{l=1}^{2n}\sum_{r=1}^n
t'{}_i^k b^{kl}t'{}_l^{n+r}b_{n+r,j}
\]
But also,  in $H(\ev_l,\ev_r)$ we have $t'{}_l^{n+r}=0$ for $l,r=1,\dots,n$, so
\[
=\sum_{k,l,r=1}^{n}
t'{}_i^k b^{k,n+l}t'{}_{n+l}^{n+r}b_{n+r,j}
\]
Similar equation of the left-axiom of the antipode, shows that
the matrix $\frak t'\in M_n(H(\ev_l,\ev_r))$ is invertible (and not only
the $(t'{}_i^j)_{i,j=1}^{2n}$-matrix in $M_{2n}(H(\ev_l,\ev_r))$.

Equation \eqref{eqid} says that the 
matrix $\mathfrak t\in M_n(A(c)[D^{-1}])$
has right inverse, but because $\frak t$
is the image of 
$\mathfrak  t'\in M_n(H(\ev_l,\ev_r))$ and $\frak t'$ is invertible
we conclude that $\frak t$ is also invertible, hence, 
it has  right inverse and it is equal to the left inverse. This is precisely
  condition \eqref{delotrolado}
\end{proof}

\section{The  (locally finite) graded case and comments
on other related  works}

\subsection{$\Z$-Graded vector spaces and graded coactions}

Assume  $W=\bigoplus_{p\in \Z}W_p$ is a graded vector space
 such that $W_p$ is finite dimensional
for every $p\in\Z$. Let
\[
C_p:=\End(W_p)^*\ \hbox{ and } C_{gr}:=\bigoplus_{p\in\Z}C_p
\]
For every $p$, fix \{$x_1^{(p)},\dots,x_{\dim W_p}^{(p)}\}$ a basis of $W_p$,
denote $\{t^{(p)}{}_ i^j\}_{i,j=1}^{\dim W_p}$ the corresponding basis of $C_p$, then $W$ is a $C_p$-comodule by
defining
\[
\rho\big(x_i^{(p)}\big):=\sum_{j=1}^{\dim W_p}t^{(p)}{}_i^j\ot x_j\]
It verifies $\rho(W_p)\subseteq C_p\ot W_p\subseteq C_{gr}\ot W_p$, 
that is, it is a graded $C_{gr}$-comodule.
Define
$T(C_{gr})$ the tensor algebra with comultiplication extending the
 comultiplication of $C_{gr}$. $W$ is also a (graded) $T(C_{gr})$-comodule,
 and hence $W^{\ot n}$ is a 
$T(C_{gr})$-comodule for any $n$.

If $F=\{f_i:V^{\ot n_i}\to V^{\ot m_i}\}_{i\i I}$ is a family of
 family of  linear maps,
the ideal  $\II_F$ 
can be defined exactly in the same way, 
 $A_{gr}(F)=TC_{gr}/\II_F$ will be a bialgebra and
$W$ will be  a graded $A_{gr}(F)$-comodule.

\begin{rem}
If $\dim V=\sum_{p\in\Z}\dim V_p<\infty$ then $C_{gr}$ then one can
 also
consider $C=\End_k(V)^*$, and $E=\{e_p:V\to V\}$ the family of 
 projectors corresponding
the direct summands
 $V_p$ (i.e. $e_pe_{q}=\delta_{p,q}e_q$, $Im(e_p)=V_p$). We have
$T(C_{gr})=TC/\II_{E}$. If $F=\{f_i:V^{\ot n_i}\to V^{\ot m_i}\}$
is a family of graded maps, then
$A_{gr}(F)$ is a quotient of $A(F)$:
\[
A(F)\twoheadrightarrow A(E\cup F)=A_{gr}(F)
\]
\end{rem}

In particular, if $\B$ a graded algebra, with unit
$u$ and multiplication $m:\B\ot \B\to\B$, one may consider 
graded comodule structures on $\B$, and they will be governed by 
$A_{gr}(u,m)$.

\begin{ex}
Let $V=k[x]/x^2$ be considered as a unital algebra with grading
$|1|=0$ and $|x|=1$. Every graded component is of
 dimension 1, so, the graded comodule structures are
necessarily of the form
\[
\rho(1)=a\ot 1\]
\[
\rho(x)=d\ot 1\]
with $a$ and $d$ group-like elements. If the unit map $u:k\to k[x]/x^2$
is colinear then $a=1$, and $A_{gr}(u,m)=k[d]\cong k[\N_0]$,
 one 
lose the differential structure
(compare with Example \ref{dg}).  However, we will see  that
in some cases the graded comodule structures may still
be  very 
interesting, since they will correspond, in the
 quadratic case, to Manin bialgebras.
\end{ex}

An easy lemma is the following:
\begin{lem} Let $\B$ be a connected associative  
graded algebra, namely
 $\B_{-n}=0$ for $n>0$ and $\B_0=k$. Denote $\B_1=V$ and assume
$\B$ is generated by $V$. That is,
$\B=TV/(R)$ where $R$ is homogeneous
 (but not necessarily concentrated in
some specific degree).
Then $A_{gr}(\B):=A_{gr}(u,m)$ ($u$ is the unit and $m$ the 
multiplication) is generated by  $C_1:=\End(V)^*$.
That is, 
the map  $T(C_1)\hookrightarrow T(C_{gr})$
induces a surjective map $T(C_1)\to A_{gr}(\B)$.
\end{lem}

\begin{proof}
If $\{x_1,\dots,x_n\}$ is  basis of $V=\B_1$, then the 
elements of the form 
$x_{i_1}\dots x_{i_p}$ generates $\B_p$.
Denote $\rho(x_i)=\sum_jt^{(1)}{}_i^j\ot x_j$. Since
$\rho:\B\to A_{gr}(\B)\ot\B$ is an algebra map,
\[
\rho(x_{i_1}\dots x_{i_p})=\sum_{j_1\dots,j_p}t^{(1)}{}_{i_1}^{j_i}
\cdots t^{(1)}{}_{ i_{p}}^{j_p}\ot x_{i_1}\cdots x_{i_p}
\]
and we see that $C_1$ generates $A_{gr}(\B)$.

\end{proof}

\subsection{The Manin bialgebra}

Recall a quadratic algebra is a $k$-algebra of the form $\Bmanin=TV/(R)$
with  $R\subseteq V^{\ot 2}$. We will assume further that $V$ is finite
 dimensional.
 In the seminal work \cite{M}, Manin define
operations $\bullet$, $\circ$ and $(-)^!$ on quadratic algebras.
He proves that given a quadratic algebra $\Bmanin=TV/(R)$,
then $\uend(\Bmanin):=\Bmanin^!\bullet \Bmanin$ is a bialgebra, $\Bmanin$ is
an $\uend(\Bmanin)$ comodule-algebra, the structure
 map $\rho:\Bmanin\to\uend(\Bmanin)\ot \Bmanin$
satisfies  $\rho(\Bmanin_p)\subseteq \uend(\Bmanin)\ot \Bmanin_p$ $(\forall p)$,
 and moreover, $\uend(\Bmanin)$ is universal
with respect to those properties (see \cite[Section 6.6]{M2}).
As a corollary we have:

\begin{prop}\label{end}
Let $\Bmanin$ be a finitely generated quadratic algebra: $\Bmanin=TV/(R)$.
If $u_\Bmanin:k\to\Bmanin$ denotes the unit map
and
$m_\Bmanin:\Bmanin\ot\Bmanin\to \Bmanin$ its multiplication,
then
\[
A_{gr}(\Bmanin):=
A_{gr}(u_\Bmanin,m_\Bmanin)=\Bmanin^!\bullet \Bmanin
\]
\end{prop}
\begin{proof}
Both algebras are generated by 
$V^*\ot V=\End_k(V)^*$ and they share the same universal property.
The unique isomorphism determined by the universal properties
is the identity on generators.
\end{proof}

\begin{rem}
In case of a quadratic algebras $B=TV/(R)$, the subspace of defining 
relations of $A_{gr}(\Bmanin)=\Bmanin^!\bullet \Bmanin
$ is clear. It would be interesting to 
expand the class of algebras $B$  where the
defining relations of  $A(B)$ (or $A_{gr}(B)$ if $B$ is graded)
can be made explicit.
\end{rem}

\subsection{$N$-homogeneous  algebras}

In \cite{Po}, the author follows Manin construction for  
$N$-homogeneous algebras. That is, if
$A=TV/R$ where $R\subseteq R^{\ot N}$ for some $N\geq 2$. Define
$R^\perp\subseteq (V^*) ^{\ot N}\cong (V^{\ot N})^*$ the annihilator of $R$ and
$A^!:=T(V^*)/R^\perp$. Notice $(A^!)^!\cong A$. Similarly he defines
the operation $\bullet$ as follows.
For  two $N$-homogeneous algebras $A=TV/(R)$ and $B=TW/(S)$
(where $R\subseteq V^{\ot N}$ and
$S\subseteq W^{\ot N}$ for the same $N$), he define
\[
A\bullet B:=
T(V\ot W)/(\tau (R\ot S))
\]
where $\tau:(V^{\ot N})\ot (W^{\ot N})\to (V\ot W)^{\ot N} $ is defined by
\[
\tau(v_1\ot \cdots \ot v_N\ot w_1\ot \cdots \ot w_N)=
(v_1\ot w_1)\ot 
(v_2\ot w_2)\ot 
\cdots \ot (v_N\ot w_N)\in (V\ot W)^{\ot N}
\]
Denoting $end(A):=A^!\bullet A$, it is a bialgebra and $A$ is a left comodule algebra over it.
 Analogus considerations
for the algebra $A^!$. Finally, he 
defines $e(A)$ as the quotient of $end(A)$
by the relations of $end(A^!)$ (or vice versa). We notice that
 in our setting, for any finitey generated
graded algebra $A$, the universal bialgebra
$A_{gr}(A)$ is defined, independently of the degree of the relations, 
and also works for homogeneous relations of eventualy 
different degrees, and the same for a pair of graded
algebras $(A,A')$
with the same set of generators. For $N$-homogeneous graded algebras,
we don't say that our approach is easier or better than
the one in \cite{Po}, but  we say that our approach is sufficiently
 general and flexible to addapt perfectly to the $N$-homogeneous 
case, and also for multigraded case, or even to the non-graded case
(if the algebras are finite dimensional).

\subsection{Path algebras}

Let $Q=(Q_0,Q_1)$ be a  finite quiver. If $Q$ has no oriented
cycles, then $kQ$ is a finite dimensional $k$-algebra and one may consider
the multiplication map $m:kQ\ot kQ\to kQ$ and
unit $u:k\to kQ$, and consequently the universal bialgebra $A(Q):=A(m,u)$.
But also, the path algebra $kQ$ is naturally graded by length of paths;
 that is, $|x_i|=0$  $\forall i\in Q_0$ 
 and $|x_\alpha|=1$ $\forall \alpha\in Q_1$.
 So, even if $Q$ happens to have cycles, if $Q$ is finite, maybe $kQ$ is not finite dimensional, but
  $kQ$ is a locally finite graded vector space, generated as algebra in degree 0 and 1. Hence, the graded version
 $A_{gr}(Q)$ is defined, and generated by $\End(V_0)^*\oplus \End(V_1)^*$
 where $V_0=k[Q_0]$ and $V_1=k[Q_1]$ are the vector 
 spaces spanned by
 $Q_0$ and $Q_1$ respectively. Since $Q_0$ and $Q_1$ are sets, the vector
 spaces $V_0$ and $V_1$ have cannonical bases $\{x_i:i\in Q_0\}$,
 $\{x_{\alpha}:\alpha\in Q_1\}$, and
 $\End(V_0)^*\oplus \End(V_1)^*$ is the coalgebra with basis
 \[
 \{t_i^j:i,j\in Q_0\}\cup
 \{t_{\alpha}^{ \beta}:\alpha, \beta \in Q_1\}\]
 and comultplication
 \[
 \Delta t_{i}^{j}=\sum_{k\in Q_0}t_{i}^{k}\ot t_{k}^{j}\]
 \[
 \Delta t_{\alpha}^{\beta}=\sum_{\gamma\in Q_1}t_{\alpha }^{\gamma}
 \ot t_{\gamma}^{\beta}\]
 
 \begin{rem}
 If one consider the universal bialgebra associated to the graded
 object $kQ$ and the multiplication map $m:kQ\to kQ\to kQ$
 (i.e. one ignores the unit of this algebra)
 then $A_{gr}(m)$ is the bialgebra generated by
 $\{t_{i}^{j}:i,j\in Q_0\}\cup
 \{t_{\alpha}^{ \beta}:\alpha, \beta \in Q_1\}$ and relations
 \[
 t_{i}^{k}t_{j}^{k}=\delta_{ik}t_{i}^{k}
 \]
 \[
 t_{i}^{t(\beta)}t_{\alpha}^{\beta}
=
\delta_{i,t(\alpha)}t_{\alpha}^{\beta}
\]
\[
t_{\alpha}^{\beta}t_{j}^{s(\beta)}
=
\delta_{j,s(\alpha)} t_{\alpha}^{\beta}
\]

($s$ and $t$ are the source and target maps $s,t:Q_1\to Q_0$)
 with comultiplication induced by
 \[
 \Delta t_{i}^{j}=\sum_{k\in Q_0}t_{i}^{k}\ot t_{k}^{j}\]
 
 \[
 \Delta t_{\alpha}^{\beta}=\sum_{\gamma\in Q_1}
 t_{\alpha}^{\gamma}\ot t_{\gamma}^{\beta}\]

 \end{rem}
 \begin{proof}
 It is straightforward form the relations defining $kQ$:
 \[
x_ix_j=\delta_{ij}x_i\ \forall i,j\in Q_0\]
 \[
x_ix_\alpha=\delta_{i,t(\alpha)}x_\alpha\ \forall i\in Q_0, \alpha\in Q_1\]
 \[
x_ \alpha x_j=\delta_{j,s(\alpha)}x_\alpha
\ \forall j\in Q_0, \alpha\in Q_1\]
and the multiplicativity of the structure map. 
As an illustration we show the second
set of relations.
From
\[
x_ix_{\alpha}=\delta_{i,t(\alpha)}x_\alpha\]
applying $\rho$ we get
\[
\sum_{j\in Q_0}\sum_{\beta\in Q_1} t_{i}^{j}t_{\alpha}^{\beta}\ot x_jx_\beta
=
\delta_{i,t(\alpha)}\sum_{\beta\in Q_1}t_{\alpha}^{\beta}\ot x_{\beta}
\]
But  $x_jx_\beta=x_\beta$ if $j=t(\beta)$ and zero otherwise, so we get
\[
\sum_{\beta\in Q_1} t_{i}^{t(\beta)}t_{\alpha}^{\beta}\ot x_\beta
=
\delta_{i,t(\alpha)}\sum_{\beta\in Q_1}t_{\alpha}^{\beta}\ot x_{\beta}
\]
Since $\{x_{\beta}\}_{\beta\in Q_1}$ are l.i. the result follows.
  \end{proof}
 Changing notaton $t_{i}^{j}\leftrightarrow y_{ij}$
 and
 $t_{\alpha}^{\beta}\leftrightarrow y_{pq}$, the universal bialgebra
 $A_{gr}(m:kQ^{\ot 2}\to kQ)$ is {\em not the same} (they consider weak bialgebras) 
 but {\em very similar} to the ones considered  in Lemma 4.2 of \cite{HWWW}.

\subsection{ $C^*$-context}

In the $C^*$-algebra context, Wang define (see \cite{W})
 a $C*$-algebra associated with a finite dimensional $C^*$-algebra.
That definition includes that a natural state is colinear.
The $C^*$-algebra definition is more restrictive,
for instance, for
the $C*$-algebra given by the group algebra $\C[G]$
of a finite group $G$, it is showed in \cite{Woro} that
 the corresponding $C^*$-Hopf algebra
is commutative, while it is a relatively
 easy exercise to show
that the universal  construction of Section 1 for the algebra
 $M_2(k)$ gives a non commutative nor cocommutative bialgebra.
Using that
$\C[S_3]\cong\C\times\C\times M_2(\C)$ we see that our construction
gives an object different from the one defined by Wang.

\subsection{Lie/Leibniz algebras}

In \cite{AM}, the authors define a commutative bialgebra
associated with a Lie or Leibniz algebra by studying the adjoint
of the functor 
\[
A\leadsto \h\ot A\]
where $A$ is a commutative $k$-algebra and
$\h$ is a fixed Leibniz (or Lie) algebra.  The bracket in the current 
algebra is given by
\[
[x\ot a,y\ot a']:=[x,y]\ot aa'
\]
The Leibniz (resp. Lie) algebra
$A\ot \h$ is called the current algebra.
 From the adjoint functor of the current algebra they define 
what they call
the {\em  universal algebra} of $\h$. This (necessary
 commutative) algebra turns out to be  
 a quotient of a polynomial algebra in $n^2$ variables ($n=\dim\h$)
that in fact is a bialgebra, with a universal property among
commutative bialgebras coacting on $\h$.
Recall that a Leibniz algebra is a generalization of a Lie algebra
in the sense that the operation is not required to be
 antisymmetric, but it satisfies a choice of the Jacobi identity:
\[
[x,[y, z]] = [[x, y], z]-[[x, z], y]\]
Notice that there is a permutation of the letters $x,y,z$, so,
if $A$ is not commutative, it is not clear how to define a Leibniz
structure on $\h\ot A$, and the point of view in \cite{AM} do not
 generalizes to noncommutative algebras.
However, from the point of view of our universal construction 
(Section 1), the non-commutative universal bialgebra coacting
on $\h$ is clear:  just consider the bracket operation as a map
$[-,-]_\h:\h^{\ot 2}\to \h$,
and $A(\h):=A([-,-]_\h)$ will give the
 universal (in general non-commutative)
bialgebra such that  $\h$ is a comodule and  $[-,-]_\h$ is colinear. The 
abelianization
\[
A(\h)_{ab}:=A(\h)/([A(\h),A(\h)])
\]
will be of course commutative, and since the ideal generated by brackets
 is always 
a bi-ideal, this is also a bialgebra, and satisfies the same universal
 property of $A(\h)$ but among commutative bialgebras. We 
conclude that
the universal commutative bialgebra constructed in \cite{AM} is the
 abelianization of $A(\h)$.
Similar comments for the Hopf envelope $H(\h)$ and
$H(\h)_{ab}$.
However, the advantage of having a noncommutative universal
bialgebra/Hopf coacting on $\h$ is clear.
For instance 
skew-derivations, (e.g. differential graded structures) are detected
by
non-commutative bialgebras.
In a similar way to the example $k[x]/(x^2)$,
the smallest non-abelian Lie algebra already gives a nontrivial
(non commutative, nor cocommutative)
universal object:

\begin{ex} Let $\h$ be the non-commutative 2-dimensional Lie algebra 
$\h=kx\oplus ky$ with antisymmetric bracket $[x,y]=x$.
Writing $C$ as the 4-dimensional coalgebra with basis $a,b,c,d$,
\[
\rho(x)=a\ot x+b\ot y,\ \
\rho(y)=c\ot x+d\ot y,\]
\[\Delta (a)=a\ot a+b\ot c,\ \Delta(b)=a\ot b+b\ot d\]
\[\Delta (c)=c\ot a+d\ot c,\ \Delta(d)=c\ot  b+d\ot d\]
The structure map on $x\ot y$ is computed
using the standard diagonal structure:
\[
\rho(x\ot y)=(a\ot x+b\ot y)(c\ot x+d\ot y)
=ac\ot (x\ot x)+ad\ot(x\ot y)+bc\ot (y\ot x)+bd\ot (y\ot y)
\]
The requirement  $(\id\ot [,])\rho(x\ot y)=\rho([x,y])$ gives
\[
ac\ot [x,x]+ad\ot [x, y]+bc\ot [y, x]+bd\ot [y,y]
=(ad-bc)\ot x
\]
\[
=\rho([x,y])=\rho(x)=a\ot x+b\ot y\]
that is, $$(ad-bc)=a, 0=b$$
or equivalently $b=0$ and $ad=a$. Before checking the other
 conditions
for colinearity using $x\ot y$, $y\ot x$ and $y\ot y$, we see that $b=0$
implies that $a$ is group-like. In the Hopf envelope $H(\h)$,
$a$ must be invertible, and $ad=a$  forces $d=1$. It is an easy
exercise that these conditions are enough to get  the bracket
colinear: the bialgebra freely generated by $a^{\pm 1}$
and $c$ with
\[
\Delta (a)=a\ot a, \Delta( c)=c\ot a+1\ot c\]
is the universal
Hopf algebra, coacting as
\[
\rho(x)=a\ot x\]
\[\rho(y)=c\ot x+1\ot y\]
If we don't require $d$ to be invertible, it i easy to check that the other conditions on the colinearity of the bracket
are $ad=da=a$, $cd=dc$. So the universal bialgebra is
\[
A(\h)=k\langle a,c,d\rangle/(ad=da=a,\ cd=dc)
\]
 with comultiplication
\[
\Delta(a)=a\ot a,\ \Delta( d)=d\ot d, \ \Delta(c)= c\ot a+d\ot c\]
and coaction
\[
\rho(x)=a\ot x,\  \rho(y)=c\ot x+d\ot y\]

\end{ex}

\begin{rem}
If $G$ is a group and $\h$ is a $G$-graded {\em Lie} algebra, because
of the antisymmetry of the bracket,  
one may assume that $G$ is abelian. But for {\em Leibniz} algebras,
grading over non-commutative groups makes perfect sense,
and grading over a non-commutative group $G$ is
the same as a coaction
over the non-commutative Hopf algebra $k[G]$.
\end{rem}

\end{document}